\documentclass{amsart}

\usepackage{amssymb}
\usepackage{amsthm}

\theoremstyle{plain}
\newtheorem{theorem}{Theorem}[section]

\newtheorem{lemma}[theorem]{Lemma}
\newtheorem{corollary}[theorem]{Corollary}
\DeclareMathOperator{\dnm}{dnm}
\newcommand{\Ball}{\underline{B}}

\newcommand{\N}{\mathbb{N}}
\newcommand{\Z}{\mathbb{Z}}                                               
\newcommand{\Q}{\mathbb{Q}}                                               
\newcommand{\R}{\mathbb{R}}                                               

\newcommand{\PSI}{\underline{\psi}}

\newcommand{\defiso}[5]{$$\begin{array}{rrcl} #1: & #2 & \longrightarrow & #3 \\ & #4 & \longmapsto & #5 \end{array}$$}
\newcommand{\x}{\mathbf{x}}
\newcommand{\y}{\mathbf{y}}
\newcommand{\z}{\mathbf{z}}

\newcommand{\dbbk}[1]{\{\!\{ #1 \}\!\}}
\newcommand{\aand}{\hbox{\qquad and \qquad}}

\begin{document}

\title{Diagonal approximation in completions of $\Q$}
\author{Matthew Palmer}
\thanks{Research supported by an EPSRC Doctoral Training Grant. }
\address{School of Mathematics, University of Bristol, Bristol, BS8 1TW}
\email{mp12500@bristol.ac.uk}

\begin{abstract}
We prove analogues of some classical results from Diophantine approximation and metric number theory (namely Dirichlet's theorem and the Duffin--Schaeffer theorem) in the setting of diagonal Diophantine approximation, i.e. approximating elements of $\R \times \Q_{p_1} \times \cdots \times \Q_{p_r}$ by elements of the diagonal embedding of $\Q$ into this space. 
\end{abstract}

\maketitle


\section{Introduction}
In this paper, we prove analogues of some theorems of Diophantine approximation and metric number theory in the new setting of diagonal approximation. As motivation, we begin with a brief overview of the results we intend to prove analogues of, and of some of the analogues which have already been proven.

The first theorem of Diophantine approximation, dating to 1842 and due to Dirichlet, states that for any real number $x$ and natural number $N$, there exists some $a \in \Z$ and $n \in \N$ with $n \leq N$ satisfying
$$\left| x - \frac{a}{n} \right| \leq \frac{1}{nN}.$$

A corollary of this says that for any \emph{irrational} $x$, there exist infinitely many \emph{coprime} $a \in \Z, n \in \N$ satisfying
$$\left| x - \frac{a}{n} \right| \leq \frac{1}{n^2}.$$

In the years subsequent to this, many similar results were proven which replaced the function $\frac{1}{n^2}$ on the right hand side of this inequality with more general approximating functions $\psi(n)$. One of the most general results proven was the Duffin--Schaeffer theorem (Thm I from \cite{DuffinSchaeffer}), which states that for any function $\psi : \N \to \R_{\geq 0}$ satisfying
$$\sum_{n=1}^{\infty} \psi(n) = \infty$$
and
\begin{equation} \label{eq:classicallimsupcondition} \limsup_{N \to \infty} \frac{\sum_{n=1}^N \frac{\psi(n) \varphi(n)}{n}}{\sum_{n=1}^N \psi(n)} > 0 \end{equation}
(where $\varphi(n)$ is the Euler totient function), the set of those $x \in [0,1]$ which satisfy 
$$\left| x - \frac{a}{n} \right| < \frac{\psi(n)}{n}$$
for infinitely many coprime $a,n$ is of Lebesgue measure $1$.

(The long-standing Duffin--Schaeffer conjecture states that this theorem should be true even without condition \eqref{eq:classicallimsupcondition}.) 

Analogues of these results have also been proven in various different setups. The two examples which will be of particular interest to us are the setup of simultaneous approximation (where we approximate $\R^k$ by $\Q^k$) and that of $p$-adic approximation (where we approximate $\Q_p$ by $\Z_p$). Versions of Dirichlet's theorem and of the Duffin--Schaeffer theorem have been proven in both of these setups. For the simultaneous version of Dirichlet's theorem, see the original paper \cite{Dirichlet1842}; for a $p$-adic version of Dirichlet's theorem, see \cite{Mahler34}. As for the Duffin--Schaeffer theorem, a $p$-adic version is proven in \cite{Haynes2010}, whereas in the simultaneous case, it has been in fact proven that the full conjecture holds for $k \geq 2$ (see \cite{PollingtonVaughan}).

The results we prove in this paper can be seen as combinations of these two settings. They also provide a quantification of the weak approximation theorem, in the same way that results from classical Diophantine approximation quantify the statement that $\Q$ is dense in $\R$.

Our aim is to find a natural method of approximating elements of $\R \times \Q_{p_1} \times \cdots \times \Q_{p_r}$ (where the $p_i$ are \emph{different} primes). As notation for this space, we will first let
$$P = \{ \infty , p_1 , \ldots, p_r \}$$
(so $P$ is a finite set of places of $\Q$, including the infinite place). This will be standard notation throughout this paper; if not stated, $p_1, \ldots, p_r$ will always refer to the finite places contained in $P$, and $r$ will always refer to the number of finite places. We will also use $P_f$ to denote the set of all places of $P$ apart from the infinite one. Then we write
$$\Q_P := \prod_{p \in P} \Q_p = \R \times \Q_{p_1} \times \cdots \times \Q_{p_r}.$$

%
%
We will use $\iota$ to denote the diagonal embedding of $\Q$ into $\Q_P$. That is to say, we define
\defiso{\iota}{\Q}{\Q_P}{q}{(q,\ldots,q).}

Then the weak approximation theorem states that $\iota(\Q)$ is dense in $\Q_P$. We want to quantify this statement.

We consider elements of $\Q$ now not as quotients from $\Z$, but as quotients from the space
$$P^{-1} \Z := \left\{ \left. n p_1^{\nu_1} \cdots p_r^{\nu_r} \ \right| \ n, \nu_1,\ldots,\nu_r \in \Z \right\}.$$

Standard elements $\beta$ and $\gamma$ of this space will always be given the decompositions
\begin{equation}
\beta = m p_1^{\mu_1} \cdots p_r^{\mu_r} \quad \hbox{and} \quad \gamma = n p_1^{\nu_1} \cdots p_r^{\nu_r},
\label{eq:decompositions}
\end{equation}
where $m$ and $n$ are each coprime to all of the $p_i$. (We call $m$ (resp. $n$) the \emph{non-$P$} part of $\beta$ (resp. $\gamma$).)

We measure the size of elements of $P^{-1}\Z$ by the \emph{level} function $\ell$, given by
$$\ell(\gamma) = \max_{p \in P} |\gamma|_p.$$ 

%

The image of $P^{-1}\Z$ under $\iota$ is a cocompact lattice in $\Q_P$, and the quotient space
$$\Q_P / \iota(P^{-1} \Z)$$
can be identified with the fundamental domain
$$[0,1) \times \Z_{p_1} \times \cdots \times \Z_{p_r},$$
which we will denote by $\Z_P$.

(To see this identification, consider $\x = (x_{\infty}, x_{p_1}, \ldots, x_{p_r}) \in \Q_P$, and write
$$x_{p_i} = q_i + z_i,$$
where $q_i \in \{p_i\}^{-1} \Z$ and $z_i \in \Z_{p_i}$. Translating $\x$ by $\iota(-q_1-\cdots-q_r)$ gives an element of
$$\R \times \Z_{p_1} \times \cdots \times \Z_{p_r},$$
and then translating by $\iota(n)$ for some suitable integer $n$ gives a representative of $\Z_P$ as required. It is easy to show this representative must be unique.)

%
%
%
%
In \S2, we use these notions to prove the following analogue of Dirichlet's theorem.

\begin{theorem}\label{thm:diagonaldirichlet} Let $\mathbf{x} \in \Q_P$ and let $N \in \N$. Then if we define a function on places of $\Q$ by
$$\dbbk{p} = \begin{array}{rl} p & \hbox{$p$ prime,} \\ 1 & p = \infty, \end{array}$$
there exists some $\beta,\gamma \in P^{-1} \Z$ with $\ell(\gamma) \leq N$ and $\gamma > 0$ satisfying 
$$|\gamma x_p - \beta|_p \leq \frac{\dbbk{p}}{N}$$
for all $p \in P$. \end{theorem}

We will also prove an analogue of the corollary to Dirichlet's theorem given above.

In the remainder of the paper, we will go on to prove a version of the Duffin--Schaeffer theorem in this setup. To state our result, we first define a set $\N_P$ by
$$N_P = \{ n \in \N \ | \ (n,p) = 1 \hbox{ for all } p \in P_f \}.$$

Then for each $p \in P$, we take a function $\psi_p : \N_P \to \R_{\geq 0}$, and restrict these such that for each $n \in \N_P$ we have
\begin{equation}
\psi_{\infty}(n) \leq \frac{1}{2n} \hbox{ and } \psi_p(n) \leq 1 \hbox{ for all } p \in P_f.
\label{eq:boundbyL}
\end{equation}

(This restriction corresponds to the (implicit) condition in Duffin and Schaeffer's original paper (\cite{DuffinSchaeffer}) that $\psi(n) \leq \frac{1}{2}$.)

We package all of the information from these functions up into one function $\PSI$, by defining
\defiso{\PSI}{\N_P}{\R_{\geq 0}^{r+1}}{n}{\oplus_{p \in P} \psi_p(n),}
where $\oplus$ denotes the direct product. We also define a function $\Psi$ by
\defiso{\Psi}{\N_p}{\R_{\geq 0}}{n}{\prod_{p \in P} \psi_p(n),}
where $\Pi$ denotes the arithmetic product.

For an element $q \in \Q$, we can always write it uniquely as
$$q = (-1)^b \frac{a}{n} p_1^{\chi_1} \cdots p_r^{\chi_r},$$
where $b \in \{0,1\}$ and $a$ and $n$ are both in $\N_P$. We call $n$ the \emph{non-$P$ denominator} of $q$, and write $\dnm_P q = n$. Then if we take any element $\x \in \Q_P$, we say that an element $q \in \Q$ forms a \emph{$\PSI$-good approximation to $\x$} if we have
$$|x_p - q|_p \leq \psi_p(\dnm_P(q))$$
for each $p \in P$. Using this definition, we can define a set $A'_P(\PSI)$ by
\begin{equation} \label{eq:definitionofA} A'_P(\PSI) = \left\{ \x \in \Z_P \ \left| \ \begin{array}{c} \hbox{there exist infinitely many $q \in \Q$ such} \\ \hbox{that $q$ is a $\PSI$-good approximation to $\x$} \end{array}  \right. \right\}. \end{equation}

Finally, we fix the measure on $\Q_P$ to be the product measure of Lebesgue measure on $\R$ and normalised Haar measures on the $\Q_{p_i}$.

Then our theorem is as follows.

\begin{theorem}\label{thm:diagonalduffinschaeffer} If we have
\begin{equation} \label{eq:diagonaldivergencecondition} \sum_{n \in \N_P} \varphi(n) \Psi(n) = \infty \qquad \end{equation}
and
\begin{equation} \label{eq:diagonallimsupcondition} \limsup_{N \to \infty} \frac{\sum_{\substack{n \in \N_P \\ n \leq N}} \varphi(n) \Psi(n)}{\sum_{\substack{n \in \N_P \\ n \leq N}} n \Psi(n)} > 0,\end{equation}
then $A'_P(\PSI)$ has measure $1$. \end{theorem}



In \S\ref{sec:duffinschaefferstatement}, we will first show a partial converse to Theorem \ref{thm:diagonalduffinschaeffer}, that convergence of the sum in \eqref{eq:diagonallimsupcondition} implies that $\mathcal{A}(\psi)$ is of measure $0$. This will follow almost directly from the convergence part of the Borel--Cantelli lemma, as in the classical case. 


In the final three sections, we first develop some of the machinery required to prove Theorem \ref{thm:diagonalduffinschaeffer}, and then conclude by proving the theorem. In $\S4$, we prove the following zero-one law, which is an analogue to Theorem 1 in \cite{Gallagher61}.


\begin{theorem}\label{thm:diagonalzeroonelaw} For each $\PSI$ satisfying \eqref{eq:boundbyL}, the set $A'_P(\PSI)$ (as defined by \eqref{eq:definitionofA}) has measure $0$ or $1$. \end{theorem}

In \S5, we prove a technical lemma (Lemma \ref{thm:diagonaloverlapestimates}), which provides estimates for the measure of the overlap between certain sets. 

Finally, in \S6 we use Theorem \ref{thm:diagonalzeroonelaw} and Lemma \ref{thm:diagonaloverlapestimates} to prove Theorem \ref{thm:diagonalduffinschaeffer}.

\textit{Acknowledgements.} The author would like to thank Alan Haynes for his helpful feedback and advice regarding this work, and Adam Morgan and Andrew Corbett for their useful comments on the paper.
\section{Dirichlet's theorem in the diagonal setting}\label{sec:diagonaldirichlet}
%
%
%
%
%
%
\begin{proof}[Proof of Theorem \ref{thm:diagonaldirichlet}] Consider the points of the form
$$\zeta \x - \iota(\beta_{\zeta}),$$
where $\zeta$ ranges over all elements of $P^{-1} \Z$ with $\ell(\zeta) \leq N$ and $\zeta \geq 0$, and the $\beta_{\zeta} \in P^{-1} \Z$ are chosen so that the points lie in $\Z_P$. This can be done uniquely since $\Z_P$ is a fundamental domain for our quotient space
$$\Q_P / \iota(P^{-1} \Z).$$

If any two of these points are equal, then taking their difference yields $\beta$ and $\gamma$ such that $\gamma \x - \beta = 0$. So we may assume they are all distinct. To apply the pigeonhole argument we want to use, we need to know how many points of this form there are. Since they are all distinct, this is equivalent to calculating the size of the set
$$Z_N = \left\{ \left. \zeta \in P^{-1} \Z \ \right| \ \ell(\zeta) \leq N, \ \zeta \geq 0 \right\}.$$

For each $i = 1, \ldots, r$, we can find $n_i \in \Z_{\geq 0}$ such that
$$p_i^{n_i} \leq N < p_i^{n_i + 1}.$$

So since $|\cdot|_p$ takes discrete values from $\{ p^m \ | \ m \in \Z \}$, we want
$$0 \leq |\zeta|_{\infty} \leq N \qquad \hbox{and} \qquad 0 \leq |\zeta|_{p_i} \leq p_i^{n_i}$$
for each $i$.

The $p_i$-adic conditions tell us we are dealing with a subset of 
$$\frac{1}{p_1^{n_1} \cdots p_r^{n_r}}\Z = \left\{ \left. \frac{n}{p_1^{n_1} \cdots p_r^{n_r}} \ \right| \ n \in \Z \right\}.$$

(Note that in this set, we only make a certain element (and its factors) invertible, and hence this set should not be confused with something of the form $P^{-1} \Z$, where we make \emph{all powers} of certain elements invertible.) Combining this with the first condition, we get
$$\# Z_N = N p_1^{n_1} \cdots p_r^{n_r} + 1.$$

%
Now consider the boxes of the form
$$\left. \left[ \frac{r}{N} , \frac{r+1}{N} \right. \right) \times (s_1 + p_1^{n_1} \Z_{p_1}) \times \cdots \times (s_r + p_r^{n_r} \Z_{p_r}),$$
where $0 \leq r \leq N - 1$ and $0 \leq s_i \leq p_i^{n_i} - 1$.

Since we have $N p_1^{n_1} \cdots p_r^{n_r}$ boxes which cover our fundamental domain, and we have
$$\{\zeta \ | \ \ell(\zeta) \leq N \} > N p_1^{n_1} \cdots p_r^{n_r},$$
there will be two elements (corresponding to $\zeta$ and $\xi$, say) in one box. But any two elements $\x,\y$ in one box must satisfy
$$|x_p - y_p|_p < \frac{\dbbk{p}}{N}$$
for each $p \in P$.

So then, assuming without loss of generality that $|\zeta|_{\infty} > |\xi|_{\infty}$, we define $\gamma = \zeta - \xi$ and $\beta = \beta_{\zeta} - \beta_{\xi}$. Then we have $\ell(\gamma) \leq N$ and
$$|\gamma x_p - \beta|_p \leq \frac{\dbbk{p}}{N}$$
for each $p \in P$ as required. \end{proof}

We now prove a corollary of this theorem, which is an analogue of the corollary to Dirichlet's theorem given in the introduction. To do this, we first need to note what it means for two elements $\beta, \gamma \in P^{-1} \Z$ to be coprime.

Let $\beta,\gamma$ be elements of this space, and decompose them as in \eqref{eq:decompositions}. Then we say that $\beta$ and $\gamma$ are coprime if $m$ and $n$ are coprime in the usual sense, and define $\gcd_P(\beta,\gamma) := \gcd(m,n)$. This definition of coprimality comes from the fact that $P^{-1} \Z$ is a UFD; by adjoining the inverses of each of the $p_i$, we have made them into units, and hence we are justified in ignoring them as factors of $\beta$ and $\gamma$.

Now we state our corollary.

\begin{corollary}\label{cor:diagonaldirichlet} Let $\mathbf{x} \in \Q_P - \iota(\Q)$. Then there exist infinitely many coprime $\beta,\gamma \in P^{-1} \Z$ such that
$$|\gamma x_p - \beta|_p \leq \frac{\dbbk{p}}{|\gamma|_p},$$
for each $p \in P$. \end{corollary}

\begin{proof}[Proof of Corollary \ref{cor:diagonaldirichlet}] For each $n \in \N$, let $B_n, C_n \in P^{-1} \Z$ be such that $\ell(C_n) \leq n$ and 
$$|C_n x_p - B_n|_p \leq \frac{\dbbk{p}}{n}$$
for each $p \in P$.

Let $\beta_n = \frac{B_n}{\gcd_P(B_n,C_n)}$ and $\gamma_n = \frac{C_n}{\gcd_P(B_n,C_n)}$. Then we have
$$|\gamma_n x_p - \beta_n|_p = \frac{1}{|\gcd_P(B_n,C_n)|_p} |C_n x_p - B_n|_p \leq |C_n x_p - B_n|_p \leq \frac{\dbbk{p}}{n} \leq \frac{\dbbk{p}}{\ell(C_n)} \leq \frac{\dbbk{p}}{\ell(\gamma_n)}.$$

So now we just need to show that of the $(\beta_n,\gamma_n)$, infinitely many are distinct.

But suppose that there are only finitely many distinct pairs
$$(\beta_{n_1},\gamma_{n_1}),\ldots,(\beta_{n_m},\gamma_{n_m}),$$
and consider
$$ C := \min_{i=1,\ldots,m} \max_{p \in P} |\gamma_{n_i} x_p - \beta_{n_i}|_p.$$

If $C = 0$, then for some $(\beta_n,\gamma_n)$ we have
$$|\gamma_n x_p - \beta_n|_p = 0$$
for all $p \in P$. But this can only happen when $\x \in \iota(\Q)$, and we assumed otherwise. 

However, if $C > 0$, then let $M$ be the maximum of the prime numbers in $P$, and then take $N$ to be some natural number with $N > \frac{M}{C}$. Then we consider $(\beta_N, \gamma_N)$. We have
$$|\gamma_N x_p - \beta_N|_p \leq \frac{\dbbk{p}}{N} < C \frac{\dbbk{p}}{M} < C$$
for each $p \in P$, which means that
$$\max_{p \in P} |\gamma_N x_p - \beta_N|_p < C,$$
giving a contradiction. \end{proof}
\section{A partial converse to Theorem \ref{thm:diagonalduffinschaeffer}}
\label{sec:duffinschaefferstatement}
Before we prove Theorem \ref{thm:diagonalduffinschaeffer}, we will prove the following result.

For each $\PSI$ satisfying \eqref{eq:boundbyL},

\begin{lemma} Let $\PSI : \N_P \to \R_{\geq 0}^{r + 1}$ satisfy \eqref{eq:boundbyL} and be such that the sum
$$\sum_{n \in \N_P} \varphi(n) \Psi(n)$$
converges. Then the set $A'_P(\psi)$ (as defined in \eqref{eq:definitionofA}) has measure $0$. \end{lemma}

The classical analogue states that if $\psi : \Z \to \R_{\geq 0}$ is such that the sum
$$\sum_{n=1}^{\infty} \frac{\varphi(n) \psi(n)}{n}$$
converges, then the set of those $x \in [0,1]$ which satisfy
$$|x - \frac{a}{n}| < \frac{\psi(n)}{n}$$
for infinitely many coprime $a,n$ is of Lebesgue measure $0$. This follows almost directly from the convergence part of the Borel--Cantelli lemma. In the diagonal case, we need to do a little more work.

\begin{proof} We note that if for each $n \in \N_P$ we define
\begin{equation} \label{eq:definitionofAn} A'_{P,n}(\PSI) = \left\{ \x \in \Z_P \ \left| \ \begin{array}{c} \hbox{there exists some $q \in \Q$ with $\dnm_P(q) = n$} \\ \hbox{such that $q$ is a $\PSI$-good approximation to $\x$} \end{array}  \right. \right\}, \end{equation}
then we can write $A'_P(\PSI)$ as a limsup set
$$A'_P(\PSI) = \limsup_{n \in \N_P} A'_{P,n}(\PSI).$$

Now we want to rewrite our sets $A'_{P,n}(\PSI)$ to make them easier to work with. 

If $q \in \Q$ has $\dnm_P q = n$, then we have
$$q = p_1^{\chi_1} \cdots p_r^{\chi_r} \frac{a}{n}$$
with $(a,p_1\cdots p_r n) = (n,p_1 \cdots p_r) = 1$. 

Since we have condition \eqref{eq:boundbyL}, the only $q \in \Q$ which can be $\PSI$-good approximations for $\x \in \Z_P$ must have $\iota(q) \in \Z_P$. So we need $q \in [0,1)$ and $q \in \Z_{p_i}$ for $i = 1, \ldots, r$. The $p$-adic conditions translate to $\chi_i \geq 0$ for $i = 1, \ldots, r$, and hence we just have
$$q = \frac{a}{n}$$
with $(a p_1 \cdots p_r,n) = 1$. Then the condition $q \in [0,1)$ means that $0 < a < n$.

So we have that
$$A'_{P,n}(\PSI) = \bigcup_{\substack{a = 1 \\ (a,n) = 1}}^n \Ball \left( \frac{a}{n} , \PSI(n) \right),$$
where the boxes $\Ball$ are defined as a direct product of balls:
$$\Ball \left( \frac{a}{n} , \PSI(n) \right) = \prod_{p \in P} B_p \left( \frac{a}{n} , \psi_p(n) \right).$$

Now we want to use this rewriting to calculate the measure of $A'_{P}(\PSI)$. We have
$$\lambda(A'_{P,n}(\PSI)) = \varphi(n) \cdot \prod_{p \in P} \lambda \left( B_p \left( \frac{a}{n} , \psi_p(n) \right) \right).$$

Then since we have
$$\lambda \left( B_{\infty} \left( \frac{a}{n} , \psi_{\infty}(n) \right) \right) = 2 \psi_{\infty}(n)$$
and 
$$\frac{\psi_p(n)}{p} < \lambda \left( B_p \left( \frac{a}{n} , \psi_p(n) \right) \right) \leq \psi_p(n),$$
we have that
$$\frac{2}{p_1 \cdots p_r} \varphi(n) \Psi(n) < \lambda(A'_{P,n}(\PSI)) \leq 2 \varphi(n) \Psi(n).$$

So if
$$\sum_{n \in \N_P} \varphi(n) \Psi(n) < \infty,$$
then we have that
$$\sum_{n \in \N_P} \lambda(A'_{P,n}(\PSI)) < \infty,$$
and hence by the convergence part of the Borel--Cantelli lemma, $A'_{P,n}(\PSI)$ has measure $0$. \end{proof}

Now we turn our attention to the proof of our main theorem.
\section{The zero-one law}
In this section, we prove Theorem \ref{thm:diagonalzeroonelaw}. For this, we will need a preliminary lemma, which is a direct analogue of Lemma 2 from \cite{Gallagher61}.

\begin{lemma}\label{lem:gallagherlem1QP} Let $\{ I_k \}$ be a sequence of boxes in $\Q_P$ such that $\lambda(I_k) \to 0$ (where a box is a product of balls from each of the constituent spaces), and let $\{U_k\}$ be a sequence of measurable sets (also in $\Q_P$) such that, for some positive $\varepsilon < 1$, we have
$$U_k \subset I_k \qquad \hbox{and} \qquad \lambda(U_k) \geq \varepsilon \lambda(I_k).$$

Then 
$$\lambda \left( \limsup_{k \to \infty} I_k \right) = \lambda \left( \limsup_{k \to \infty} U_k \right).$$ \end{lemma}

\begin{proof} We define
$$\mathcal{I} = \bigcap_{L=1}^{\infty} \bigcup_{k \geq L} I_k \left( = \limsup_{k \to \infty} I_k \right),$$
$$\mathcal{U}_M = \bigcup_{k \geq M} U_k,$$
$$\mathcal{D}_M = \mathcal{I} - \mathcal{U}_M.$$

Then the lemma can be restated as
$$\lambda \left( \bigcup_{M=1}^{\infty} \mathcal{D}_M \right) = 0.$$

We prove that each individual $\mathcal{D}_M$ has measure $0$. In order to do this, we need to define the notion of a density point of a set $A \subseteq \Q_P$.

For $\x \in \Q_P$, define
$$d_{\varepsilon}(\x) := \frac{\lambda(A \cap B_{\varepsilon}(\x))}{\lambda(B_{\varepsilon}(x))},$$
where
$$\underline{B}(\x,\varepsilon) = \prod_{p \in P} B_{p}(x_p,\varepsilon).$$

Then we say $\x$ is a density point of $A$ if $\lim_{\varepsilon \to 0} d_{\varepsilon}(\x)$ exists and is equal to $1$.

By the Lebesgue density theorem, and its analogue in $\Q_p$, almost all points $\x$ are density points. (For the $p$-adic analogue, see Theorem I in \cite{PopkenTurkstra} or page 14 of \cite{Lutz}.)

So now suppose for a contradiction that $\x_0$ is a density point of $\mathcal{D}_M$ in $\mathcal{D}_M$.

Firstly, since we have that $\x_0 \in I_k$ for infinitely many $k$ (by the definition of $\mathcal{D}_M$) and that $\lambda(I_k) \to 0$, if we restrict to those $k$ such that $\x_0 \in I_k$, we have
$$\lambda(\mathcal{D}_M \cap I_k) \sim \lambda(I_k)$$
as $k \to \infty$ (since $\x_0$ is a density point of $\mathcal{D}_M$).

However, we also have that $\mathcal{D}_M \cap U_k = \emptyset$ for any $k \geq M$, and hence $U_k$ and $\mathcal{D}_M \cap I_k$ are disjoint subsets of $I_k$. From this, we get
$$\lambda(I_k) \geq \lambda(U_k) + \lambda(\mathcal{D}_M \cap I_k) \geq \varepsilon \lambda(I_k) + \lambda(\mathcal{D}_M \cap I_k),$$
and hence
$$\lambda(\mathcal{D}_M \cap I_k) \leq (1 - \varepsilon) \lambda(I_k),$$
contradicting our first part. \end{proof}

We also want to prove another result, namely the following.

\begin{lemma}\label{lem:gallagherlem2QP} Let $q,s \in \Z$ be such that $q > 1$, and define maps $T_{\infty}, T_{p_1}, \ldots, T_{p_r}$ by
$$T_{\infty} : x \mapsto q\x + \iota \left( \frac{s}{q} \right) \mod \iota \left( P^{-1}\Z \right), \quad T_{p_i} : \x \mapsto \frac{1}{p_i} \x \mod \iota \left( P^{-1} \Z \right)$$
which send $\Z_P$ to itself. Then if a set $A \subseteq \Z_P$ satisfies $T_{\infty}(A) \subseteq A$ and $T_{p_i}(A) \subseteq A$ for $i=1,\ldots,r$, then $A$ has measure $0$ or $1$. \end{lemma}
\begin{proof}[Proof of Lemma \ref{lem:gallagherlem2QP}] Suppose that $A$ is of positive measure. Then $A$ has a density point $\y$. By the definition, for any $\delta > 0$, we can find an $E > 0$ such that for all $B_P(\y, \varepsilon)$ with $\varepsilon < E$, we have
$$\frac{\lambda_P(A \cap B_P(\y,\varepsilon))}{\lambda_P(B_P(\y,\varepsilon))} \geq 1 - \delta.$$

For each $\delta > 0$, consider such an $\varepsilon$. Then for each $p_i$, take $m_i$ to be the unique integer such that
$$p_i^{-m_i} \leq \varepsilon < p_i^{-m_i + 1}.$$

Then we have
$$B_P(\y,\varepsilon) = B_{\infty}(y_{\infty},\varepsilon) \times \prod_{p_i} (y_{p_i} + p_i^{m_i} \Z_{p_i}).$$

We have 
$$T_{p_r}^{m_r}(\cdots(T_{p_1}^{m_1}(B_P(\y,\varepsilon)))) = B_{\infty}(z_{\infty}, \varepsilon p_1^{-m_1} \cdots p_r^{-m_r}) \times \prod_{p_i} (z_{p_i} + \Z_{p_i})$$
for some $\z = (z_{\infty}, z_{p_1}, \ldots, z_{p_r}) \in \Z_P$. So now take $m_{\infty}$ such that
$$1 \leq q^{m_{\infty}} \varepsilon p_1^{-m_1} \cdots p_r^{-m_r} < q,$$
and define
$$T = T_{\infty}^{m_{\infty}} \circ T_{p_r}^{m_r} \circ \cdots \circ T_{p_1}^{m_1}.$$

Then some translate of $T(B_P(\y,\varepsilon))$ by an element of $\iota(P^{-1}(\Z))$ completely covers $\Z_P$, and we have
$$1 \leq \lambda_P(T(B_P(\y,\varepsilon))) < q.$$

The map $T$ expands the measure of sets by $q^{m_{\infty}}$, and hence we have
$$\frac{\lambda_P(T(A \cap B_P(\y,\varepsilon)))}{\lambda_P(T(B_P(\y,\varepsilon)))} = \frac{q^{m_{\infty}} \lambda_P(A \cap B_P(\y,\varepsilon))}{q^{m_{\infty}} \lambda_P(B_P(\y,\varepsilon))} \geq 1 - \delta.$$

So
$$\lambda_P(T(A \cap B_P(\y,\varepsilon))) \geq (1 - \delta) \lambda_P(T(B_P(\y,\varepsilon))).$$

We have
$$T(A \cap B_P(\y,\varepsilon)) = T(A) \cap T(B_P(\y,\varepsilon)),$$
and hence the difference between $T(A) \cap T(B_P(\y,\varepsilon))$ and $T(B_P(\y,\varepsilon))$ has at most measure
$$\delta \lambda_P(T(B_P(\y,\varepsilon))).$$

Then since we know that
$$1 \leq \lambda_P(T(B_P(\y,\varepsilon))) < q,$$
the difference has measure at most $q \delta$.

We have
$$T(A) \cap T(B_P(\y,\varepsilon)) \subseteq T(A) \subseteq A,$$
and we know that $T(B_P(\y,\varepsilon))$ covers $\Z_P$. So the difference between $A$ and $\Z_P$ has at most measure $q \delta$. Taking $\delta \to 0$ completes the proof. \end{proof}

Now we are in a position to prove our zero-one law. 

\begin{proof}[Proof of Theorem \ref{thm:diagonalzeroonelaw}] For each prime $\pi > p_1 \cdots p_r$ and for each $\nu \in \N$, we define sets $\mathfrak{A}(\pi^{\nu})$ and $\mathfrak{B}(\pi^{\nu})$ by
$$\mathfrak{A}(\pi^{\nu}) = \left\{ \x \in \Z_P \ \left|  \begin{array}{c} \hbox{there exist infinitely many $q \in \Q$ with $\pi \nmid \dnm_P(q)$} \\ \hbox{such that $q$ is a $\pi^{\nu-1} \PSI$-good approximation to $\x$} \end{array} \right. \right\}$$
and
$$\mathfrak{B}(\pi^{\nu}) = \left\{ \x \in \Z_P \ \left|  \begin{array}{c} \hbox{there exist infinitely many $q \in \Q$ with $\pi \mid\mid \dnm_P(q)$} \\ \hbox{such that $q$ is a $\pi^{\nu-1} \PSI$-good approximation to $\x$} \end{array} \right. \right\}$$
where $a \mid\mid b$ means that $a$ divides $b$ \emph{exactly} (i.e. $a \mid b$ but $a \nmid b$).

Note that both $\mathfrak{A}(\pi)$ and $\mathfrak{B}(\pi)$ are subsets of $A'_P(\PSI)$, and also note that we have
$$\mathfrak{A}(\pi) \subseteq \mathfrak{A}(\pi^2) \subseteq \mathfrak{A}(\pi^3) \subseteq \cdots$$
and 
$$\mathfrak{B}(\pi) \subseteq \mathfrak{B}(\pi^2) \subseteq \mathfrak{B}(\pi^3) \subseteq \cdots.$$

The set $\mathfrak{A}(\pi)$ can be written as a limsup of boxes of measure $\sim \Psi(n)$, where $n \to \infty$. Then, since we assumed that
$$\Psi(n) \leq \frac{1}{2n},$$
the measures tend to zero, and we can apply Lemma \ref{lem:gallagherlem1QP} to give us that
$$\lambda_P(\mathfrak{A}(\pi^{\nu})) = \lambda_P(\mathfrak{A}(\pi))$$
for all $\nu \in \N$. Hence (since they form a chain) the \emph{union} $\mathfrak{A}^*(\pi)$ of all of the $\mathfrak{A}(\pi^{\nu})$ must also have measure $\lambda_P(\mathfrak{A}(\pi))$.

The same argument gives us that the union $\mathfrak{B}^*(\pi)$ of all of the $\mathfrak{B}(\pi^{\nu})$ has measure $\lambda_P(\mathfrak{B}(\pi))$. 

Now we want to construct maps $T_A$ and $T_B$ such that
$$T_A(\mathfrak{A}^*(\pi)) \subseteq \mathfrak{A}^*(\pi) \aand T_B(\mathfrak{B}^*(\pi)) \subseteq \mathfrak{B}^*(\pi).$$

Let $\x \in \mathfrak{A}(\pi^{\nu})$. Then there are infinitely many $q$ with $\pi\nmid\dnm_P(q)$ such that $q$ is a $\pi^{\nu-1} \PSI$-good approximation to $\x$. For each such $q$, we have
$$ |x_p - q|_p \leq \pi^{\nu-1} \psi_p(\dnm_P(q))$$
for each $p \in P$. Now consider the element $\frac{\pi}{p_1 \cdots p_r} \x$. We have
\begin{align*} \left| \frac{\pi}{p_1 \cdots p_r} x_p - \frac{q \pi}{p_1 \cdots p_r} \right|_p &= \left| \frac{\pi}{p_1 \cdots p_r} \right|_p |x_p - q|_p \\ &\leq \left| \frac{\pi}{p_1 \cdots p_r} \right|_p \pi^{\nu-1} \psi_p(\dnm_P(q)) \\ &\leq \pi^{\nu} \psi_P(\dnm_P(q)), \end{align*}
where the last inequality comes from the fact that 
$$\left| \frac{\pi}{p_1 \cdots p_r} \right|_p < \pi$$
for each $p \in P$. The element $\frac{q \pi}{p_1 \cdots p_r}$ also has the same non-$P$ denominator as $q$. So we have that 
$$\frac{\pi}{p_1 \cdots p_r} \x \in \mathfrak{A}(\pi^{\nu + 1}),$$
and hence the map
$$\x \mapsto \frac{\pi}{p_1 \cdots p_r} \x \mod \iota(P^{-1}\Z)$$
sends $\mathfrak{A}(\pi^{\nu})$ to $\mathfrak{A}(\pi^{\nu+1})$, therefore sending $\mathfrak{A}^*(\pi)$ into itself. Denote this map by $T_A$.

By a very similar argument, we can show that the map $T_B$ given by
$$x \mapsto \frac{\pi}{p_1 \cdots p_r} \x + \frac{p_1 \cdots p_r}{\pi} \mod \iota(P^{-1}\Z)$$
sends $\mathfrak{B}^*(\pi)$ into itself.

Next, we can apply Lemma \ref{lem:gallagherlem2QP} to show us that both $T_A$ and $T_B$ are metrically transitive. So for any prime $\pi > p_1 \cdots p_r$, we have that $\mathfrak{A}^*(\pi)$ and $\mathfrak{B}^*(\pi)$ are measure $0$ and $1$, and hence so are $\mathfrak{A}(\pi)$ and $\mathfrak{B}(\pi)$. 

Since $\mathfrak{A}(\pi)$ and $\mathfrak{B}(\pi)$ are subsets of $A'_P(\PSI)$, if either of them is measure $1$ for any prime $\pi > p_1 \cdots p_r$, we must have that $A'_P(\PSI)$ is also measure $1$. So now we tackle the only remaining case, which is where both $\mathfrak{A}(\pi)$ and $\mathfrak{B}(\pi)$ are measure $0$ for all primes $\pi > p_1 \cdots p_r$.

For each of those primes, define a set $\mathfrak{C}(\pi)$ by
$$\mathfrak{C}(\pi) = \left\{ \x \in \Z_P \ \left|  \begin{array}{c} \hbox{there exist infinitely many $q \in \Q$ with $\pi^2 \mid \dnm_P(q)$} \\ \hbox{such that $q$ is a $\PSI$-good approximation to $\x$} \end{array} \right. \right\}.$$

Since we assumed that $\mathfrak{A}(\pi)$ and $\mathfrak{B}(\pi)$ both have measure $0$, we have
$$\lambda(A'_P(\PSI)) = \lambda(\mathfrak{C}(\pi))$$
for each $\pi$. So now suppose that $A'_P(\PSI)$ has measure $>0$. We want to use this to show that it has measure $1$.

By the Lebesgue density theorem, $A'_P(\PSI)$ must have a density point
$$\x = (x_{\infty}, x_{p_1}, \ldots, x_{p_r}).$$

That is to say, there exists some $\x \in A'_P(\PSI)$ such that for each $\delta > 0$, there exists an $E > 0$ such that for all $\varepsilon \leq E$, we have that
$$\frac{\lambda(A'_P(\PSI) \cap B_P(\x,\varepsilon))} {\lambda(B_P(\x,\varepsilon))} > 1 - \delta.$$

For each $\delta$, take $\varepsilon = E$.

Since $A'_P(\PSI)$ and $\mathfrak{C}(\pi)$ differ by a set of measure $0$, we also have that
$$\frac{\lambda(\mathfrak{C}(\pi) \cap B_P(\x, \varepsilon))} {\lambda(B_P(\x, \varepsilon))} > 1 - \delta$$
for all $\pi$.

For each $\pi$, we have that $\mathfrak{C}(\pi)$ is periodic by $\iota(\frac{1}{\pi})$. So if we define a set
$$U_{\varepsilon}(\x,\pi) = \bigcup_{a \in \Z} B_P \left( \x + \iota \left( \frac{a}{\pi} \right) , \varepsilon \right),$$
then we have
$$\frac{\lambda(\mathfrak{C}(\pi) \cap U_{\varepsilon}(\x,\pi))} {\lambda(U_{\varepsilon}(\x,\pi) \cap \Z_P)} > 1 - \delta.$$

Now we want to show that for each $\varepsilon$, there exists some prime $\pi$ such that
$$\Z_P \subseteq U_{\varepsilon}(\x,\pi).$$

If this is the case, then we will have
$$\mathfrak{C}(\pi) \cap U_{\varepsilon}(\x,\pi) = \mathfrak{C}(\pi) \hbox{\quad and \quad} U_{\varepsilon}(\x,\pi) \cap \Z_P = \Z_P,$$
giving us
$$\lambda(A'_P(\PSI)) = \lambda(\mathfrak{C}(\pi)) > 1 - \delta.$$

Then since we can take $\delta$ arbitrarily small, we will have our result.

For each $p_i$, there exists some $m_i$ such that we have
$$p_i^{m_i} \leq \varepsilon < p_i^{m_i + 1}.$$

Then we have
$$B_P(\y,\varepsilon) = (y_{\infty} - \varepsilon, y_{\infty} + \varepsilon) \times \prod_{p \in P} (y_p + p^m \Z_p).$$

Let $\pi > p_1 \cdots p_r \varepsilon^{-r-1}$. We then want to show that any $\y \in \Z_P$ lies in
$$B_P \left( \x + \iota \left( \frac{a}{\pi} \right) , \varepsilon \right)$$
for some $a \in \Z$. We have
\begin{align*} B_P \left( \x + \iota \left( \frac{a}{\pi} \right) \right) &= \left( x_{\infty} + \frac{a}{\pi} - \varepsilon, x_{\infty} + \frac{a}{\pi} + \varepsilon \right) \\ & \qquad \times \prod_{p \in P} \left( \left( x_p + \frac{a}{\pi} \right) + p^m \Z_p \right). \end{align*}

Consider the real component $y_{\infty}$ of $\y$. To have $\y \in B_P(\x + \iota(\frac{a}{\pi}),\varepsilon)$, we certainly need to have
$$y_{\infty} \in \left( x_{\infty} + \frac{a}{\pi} - \varepsilon, x_{\infty} + \frac{a}{\pi} + \varepsilon \right),$$
which can be transformed to the condition
$$\pi(y_{\infty} - x_{\infty} - \varepsilon) < a < \pi(y_{\infty} - x_{\infty} + \varepsilon).$$

Since $\pi > p_1 \cdots p_r \varepsilon^{-r-1}$, this interval has length
$$> p_1 \cdots p_r \varepsilon^{-r} > p_1^{m_1} \cdots p_r^{m_r},$$
and hence there are at least $p_1^{m_1} \cdots p_r^{m_r}$ consecutive values of $a$ such that
$$y_{\infty} \in \left( x_{\infty} + \frac{a}{\pi} - \varepsilon , x_{\infty} + \frac{a}{\pi} + \varepsilon \right).$$

Now, consider
$$\left(x_{p_1} + \frac{a}{\pi} , \ldots, x_{p_r} + \frac{a}{\pi}\right) \mod p_1^{m_1} \Z_{p_1} \times \cdots \times p_r^{m_r} \Z_{p_r}$$
for exactly $p_1^{m_1} \cdots p_r^{m_r}$ consecutive values of $a$. It is easy to see that these must all be distinct, and hence for one of the $a$ such that
$$y_{\infty} \in \left( x_{\infty} + \frac{a}{\pi} - \varepsilon , x_{\infty} + \frac{a}{\pi} + \varepsilon \right),$$
we also have that
$$y_{p_i} \in \left( x_{p_i} + \frac{a}{\pi} \right) + p_i^{m_i} \Z_{p_i}$$
for $i=1,\ldots,r$. So there exists some $a \in \Z$ such that
$$\y \in B_P \left( \x + \iota \left( \frac{a}{\pi} \right) , \varepsilon \right)$$
as required, and hence we have our theorem. \end{proof}

\section{Overlap estimates}
In this section, we prove the following result, which is analogous to Lemma II in \cite{DuffinSchaeffer}.

\begin{lemma} \label{thm:diagonaloverlapestimates} Let
$$A'_{m,P}(\PSI) = \bigcup_{\substack{a=1 \\ (a,m) = 1}}^m \prod_{p \in P}  B_p \left( \frac{a}{m} , \psi_p(m) \right),$$
$$A'_{n,P}(\PSI) = \bigcup_{\substack{b=1 \\ (b,n) = 1}}^n \prod_{p \in P} B_p \left( \frac{b}{n} , \psi_p(n) \right).$$

Then for $m \neq n$ we have
$$\lambda_P(A'_{m,P}(\PSI) \cap A'_{n,P}(\PSI)) \leq 2^{r+2} mn \Psi(m) \Psi(n).$$ \end{lemma}

\begin{proof} For brevity, we write just $A_n$ for $A'_{n,P}(\PSI)$.

To get an upper bound for the measure of $A_m \cap A_n$, we note that each set is made up of a union of disjoint boxes. (This disjointness comes from our assumptions that $\psi_p(n) < 1$ and $\psi_{\infty}(n) < \frac{1}{2n}$.) We then sum over all pairs of boxes which intersect, with the summand being an upper bound for the measure of their intersection.

A pair of boxes will intersect if and only if their constituent intervals for each place in $P$ intersect. So for the overlap between
$$\prod_{p \in P} B_p \left( \frac{a}{m} , \PSI(m) \right) \hbox{\qquad and \qquad} \prod_{p \in P} B_p \left( \frac{b}{n} , \PSI(n) \right)$$
to be of positive measure, we want the real overlap
$$B_{\infty} \left( \frac{a}{m} , \psi_{\infty}(m) \right) \cap B_{\infty} \left( \frac{b}{n} , \psi_{\infty}(n) \right)$$
to be of positive measure and for each $p_i$-adic overlap
$$B_{p_i} \left( \frac{a}{m} , \psi_{p_i}(m) \right) \cap B_{p_i} \left( \frac{b}{n} , \psi_{p_i}(n) \right)$$
to also be of positive measure.

The real intervals will definitely overlap when
$$\left| \frac{a}{m} - \frac{b}{n} \right| \leq 2 \max \{ \psi_{\infty}(m) , \psi_{\infty}(n) \},$$
and the measure of their overlap will be at most 
$$ 2 \min \{ \psi_{\infty}(m) , \psi_{\infty}(n) \} $$
(since the worst-case scenario is that one interval is completely contained inside the other). 

Similarly, the $p_i$-adic intervals will overlap when
$$\left| \frac{a}{m} - \frac{b}{n} \right|_{p_i} \leq \max \{ \psi_{p_i}(m) , \psi_{p_i}(n) \},$$
and the measure of their overlap will be at most
$$\min \{ \psi_{p_i}(m) , \psi_{p_i}(n) \}.$$

So if we write
\begin{align*} \Delta_{\infty} &= 2 \max \{ \psi_{\infty}(m), \psi_{\infty}(n) \}, \\ \delta_{\infty} &= 2 \min \{ \psi_{\infty}(m) , \psi_{\infty}(n) \}, \\ \Delta_{p_i} &= \max \{ \psi_{p_i}(m) , \psi_{p_i}(n) \}, \\ \delta_{p_i} &= \min \{ \psi_{p_i}(m) , \psi_{p_i}(n) \}, \end{align*}
then we have
$$\lambda(A_m \cap A_n) \leq \left( \prod_{p \in P} \delta_p \right) N(m,n),$$
where
$$N(m,n) := \# \left\{ (a,b) \ \left| \begin{array}{c} 1 \leq a \leq m, \ 1 \leq b \leq n \\ \left| \frac{a}{m} - \frac{b}{n} \right|_p \leq \Delta_p \hbox{ for all } p \in P \end{array}  \right. \right\}.$$

So now we want to estimate $N(m,n)$. Since $|\cdot|_{p_i}$ only takes values which are powers of $p_i$, for each $p_i$ we find the unique $\tau_i \in \Z$ such that 
$$p_i^{-\tau_i} \leq \Delta_{p_i} < p_i^{1 - \tau_i},$$
and then consider the individual cases
$$\left| \frac{a}{m} - \frac{b}{n} \right|_{p_i} = p_i^{-t_i}$$
for $t_i \geq \tau_i$.

We first note that, since $m$ and $n$ are coprime to all of the $p_i$, this equation is equivalent to
$$| an - bm|_{p_i} = p_i^{-t_i}.$$

If we have
$$|an-bm|_{p_i} = p_i^{-t_i}$$
for $i=1,\ldots,r$, then we must have that
$$an - bm = \pm p_1^{t_1} \cdots p_r^{t_r} k$$
for some $k \in \N_P$. So we get
$$N(m,n) = \sum_{t_1 \geq \tau_1} \cdots \sum_{t_r \geq \tau_r} \sum_{\substack{k=1 \\ p_i \nmid k \hbox{ for} \\ i= 1,\ldots,r}}^{\frac{mn \Delta_{\infty}}{p_1^{t_1} \cdots p_r^{t_r}}} \# \left\{ (a,b) \left| \begin{array}{c} 1 \leq a \leq m, \ 1 \leq b \leq n \\ |an - bm|_{\infty} = p_1^{t_1} \cdots p_r^{t_r} k \end{array} \right. \right\}.$$

Now we use a standard lemma from elementary number theory:

\begin{lemma}\label{lem:elementary} Suppose that $m,n \in \N$ and $x \in \Z$. Then the equation
$$an - bm = x$$
has solutions $a,b \in \Z$ if and only if $(m,n) \mid x$. If $(a_0,b_0)$ is a particular solution, then the set of all solutions is
$$\left\{ \left. \left( a_0 + \frac{\ell m}{(m,n)}, b_0 + \frac{\ell n}{(m,n)} \right)  \right| \ell \in \Z \right\}.$$ \end{lemma}

Using this, we see that there are at most $(m,n)$ solutions in the range $1 \leq a \leq m, 1 \leq b \leq n$ to 
$$|an - bm|_{\infty} = p_1^{t_1} \cdots p_r^{t_r} k,$$
and that solutions only exist if we have
$$(m,n) \mid p_1^{t_1} \cdots p_r^{t_r} k.$$

So 
\begin{align*} \sum_{\substack{k = 1 \\ p_i \nmid k \ \hbox{\scriptsize for} \\ i = 1,\ldots,r}}^{\frac{mn \Delta_{\infty}}{p_1^{t_1} \cdots p_r^{t_r}}} \# \left\{ (a,b) \ \left| \ \begin{array}{c} 1 \leq a \leq m, \ 1 \leq b \leq n, \\ |an - bm|_{\infty} = p_1^{t_1} \cdots p_r^{t_r} k \end{array}  \right. \right\} &\leq \sum_{\substack{k = 1 \\ p_i \nmid k \ \hbox{\scriptsize for} \\ i = 1,\ldots,r \\ (m,n) \mid p_1^{t_1} \cdots p_r^{t_r} k}}^{\frac{mn \Delta_{\infty}}{p_1^{t_1} \cdots p_r^{t_r}}} (m,n) \\ &\leq (m,n) \sum_{\substack{k = 1 \\ (m,n) \mid p_1^{t_1} \cdots p_r^{t_r} k}}^{\frac{mn \Delta_{\infty}}{p_1^{t_1} \cdots p_r^{t_r}}} 1 \\ &\leq (m,n) \frac{\frac{mn \Delta_{\infty}}{p_1^{t_1} \cdots p_r^{t_r}}}{(m,n)} \\ &= \frac{mn \Delta_{\infty}}{p_1^{t_1} \cdots p_r^{t_r}}, \end{align*}
and therefore
\begin{align*} N(m,n) &\leq mn \Delta_{\infty} \sum_{t_1 \geq \tau_1} \cdots \sum_{t_r \geq \tau_r} \frac{1}{p_1^{t_1} \cdots p_r^{t_r}} \\ &\leq mn \Delta_{\infty} \prod_{i=1}^r \left( \sum_{t_i \geq \tau_i} \frac{1}{p_i^{t_i}} \right) \leq 2^r mn \Delta_{\infty} \prod_{i=1}^r p_i^{-\tau_i} \leq  2^r mn \Delta_{\infty} \prod_{i=1}^r \Delta_{p_i}. \end{align*}

So we get
\begin{align*} \lambda_P(A_m \cap A_n) &\leq \left( \prod_{p \in P} \delta_p \right) \left( 2^r mn \prod_{p \in P} \Delta_{p} \right) \\ &= 2^r mn \prod_{p \in P} \Delta_p \delta_p \\ &= 2^r mn \cdot 4 \prod_{p \in P} \psi_p(m) \psi_p(n) \\ &= 2^{r+2} mn \Psi(m) \Psi(n)\end{align*}
as required. \end{proof}
\section{Proof of Theorem \ref{thm:diagonalduffinschaeffer}}
In \S\ref{sec:duffinschaefferstatement}, we defined sets $A'_{P,n}(\PSI)$ such that
$$A'_P(\PSI) = \limsup_{n \in \N_P} A'_{P,n}(\PSI),$$
and showed that these sets could be written as
$$A'_{n,P}(\PSI) = \bigcup_{\substack{a=1 \\ (a,n) = 1}}^n \prod_{p \in P} B_p \left( \frac{a}{n} , \psi_p(n) \right).$$

Now we need to show that
$$\limsup_{n \in \N_P} A'_{P,n}(\PSI)$$
has measure $1$. Since Theorem \ref{thm:diagonalzeroonelaw} states that $A'_P(\PSI)$ has either measure $0$ or measure $1$, we only need to show that this set has \emph{positive} measure.

We use a lemma (Lemma 2.3 from \cite{Harman}, which we quote below) to get a lower bound on the size of our limsup set.

\begin{lemma}\label{lem:harmanlemma} Let $X$ be a measure space with measure $\lambda$ such that $\lambda(X)$ is finite. Let $\mathcal{E}_n$ be a sequence of measurable subsets of $X$ such that
$$\sum_{n=1}^{\infty} \lambda(\mathcal{E}_n) = \infty.$$

Then the set $E$ of points belonging to infinitely many sets $\mathcal{E}_n$ satisfies
$$\lambda(E) \geq \limsup_{N \to \infty} \left( \sum_{n=1}^N \lambda(\mathcal{E}_n) \right)^2 \left( \sum_{m,n=1}^N \lambda(\mathcal{E}_m \cap \mathcal{E}_n) \right)^{-1}.$$ \end{lemma}

We note that $\Z_P$ is such a measure space, and we consider our measurable subsets $A'_{P,n}(\PSI)$. We have that
$$\frac{2 \varphi(n) \Psi(n)}{p_1 \cdots p_r} < \lambda(A'_{P,n}(\PSI)) \leq 2 \varphi(n) \Psi(n),$$
and hence if
$$\sum_{n \in \N_P} \varphi(n) \Psi(n) = \infty,$$
then we have that
$$\sum_{n \in \N_P} \lambda(A'_{P,n}(\PSI)) = \infty$$
as well. Then by Lemma \ref{lem:harmanlemma} we have that
$$\lambda(\limsup A'_{P,n}(\PSI)) \geq \limsup_{N \to \infty} \left( \sum_{\substack{n \in \N_P \\ n \leq N}} \lambda(A'_{P,n}(\PSI)) \right)^2 \left( \sum_{\substack{m,n \in \N_P \\ m,n \leq N}} \lambda(A'_{P,m}(\PSI) \cap A'_{P,n}(\PSI)) \right)^{-1}.$$

So now we need to show that this limsup is positive. We can do this by appealing to the overlap estimates from Lemma \ref{thm:diagonaloverlapestimates}. Using this, we get
\begin{align*} \frac{\left(\sum_{\substack{n \in \N_P \\ n \leq N}} \lambda(A'_{P,n}(\PSI)) \right)^2}{\sum_{\substack{m,n \in \N_P \\ m,n \leq N}} \lambda(A'_{P,m}(\PSI) \cap A'_{P,n}(\PSI))} &\geq \frac{ \left( \sum_{\substack{n \in \N_P \\ n \leq N}} \frac{2 \varphi(n) \Psi(n)}{p_1 \cdots p_r} \right)^2}{C \sum_{\substack{m,n \in \N_P \\ m,n \leq N}} m n \Psi(m) \Psi(n)} \\ &\geq C \frac{\left( \sum_{\substack{n \in \N_P \\ n \leq N}} \varphi(n) \Psi(n) \right)^2}{\left( \sum_{\substack{n \in \N_P \\ n \leq N}} n \Psi(n) \right)^2} \\ &= C \left( \frac{\sum_{\substack{n \in \N_P \\ n \leq N}} \varphi(n) \Psi(n)}{\sum_{\substack{n \in \N_P \\ n \leq N}} n \Psi(n)} \right)^2,\end{align*}
and hence since we assumed \eqref{eq:diagonallimsupcondition}, we have
$$\limsup_{N \to \infty} \left(\sum_{\substack{n \in \N_P \\ n \leq N}} \lambda_P(A'_{P,n}(\PSI)) \right)^2 \left(\sum_{\substack{m,n \in \N_P \\ m,n \leq N}} \lambda_P(A'_{P,m}(\PSI) \cap A'_{P,n}(\PSI)) \right)^{-1} > 0,$$
and hence $\lambda(A'_P(\PSI)) = 1$ as required.



\end{document}